\documentclass[reqno]{amsart}
\usepackage{amsmath}
\usepackage{arydshln}%
\allowdisplaybreaks[3]
\numberwithin{equation}{section}
\numberwithin{equation}{section}

\def\proof{\indent{\em Proof.\quad}}
\def\endproof{\hfill\hbox{$\sqcup$}\llap{\hbox{$\sqcap$}}\medskip}
\newtheorem{thm}{\indent Theorem}[section]
\newtheorem{cor}[thm]{\indent Corollary}
\newtheorem{lem}[thm]{\indent Lemma}

\newtheorem{rmk}{{\indent\bf Remark}}[section]

\newcommand{\td}{\tilde}

\newcommand{\fr}{\frac}

\newcommand{\edd}{\end{document}}
\newcommand{\be}{\begin{equation}}
\newcommand{\ee}{\end{equation}}

\newcommand{\lagl}{\langle}
\newcommand{\ragl}{\rangle}
\newcommand{\lmx}{\left(\begin{matrix}}
\newcommand{\rmx}{\end{matrix}\right)}
\newcommand{\ldt}{\left|\begin{matrix}}
\newcommand{\rdt}{\end{matrix}\right|}

\newcommand{\dv}{{\rm div}}

\newcommand{\sgn}{{\rm Sgn}}

\newcommand{\eps}{\epsilon}
\newcommand{\veps}{\varepsilon}
\newcommand{\bbr}{{\mathbb R}}
\newcommand{\const}{{\rm const}}

\newcommand{\ba}{\begin{array}}
\newcommand{\ea}{\end{array}}
\newcommand{\nnm}{\nonumber}
\newcommand{\beal}{\begin{align}}
\newcommand{\eal}{\end{align}}
\newcommand{\bea}{\begin{eqnarray}}
\newcommand{\eea}{\end{eqnarray}}

\newcommand{\supp}{{\rm Supp\,}}

\begin{document}

\title[Rigidity theorems of the space-like $\lambda$-hypersurfaces in the Lorentzian space]{Rigidity theorems of the space-like $\lambda$-hypersurfaces\\in the Lorentzian space $\bbr^{n+1}_1$} 

\author[X. X. Li]{Xingxiao Li$^*$} 

\author[X. F. Chang]{Xiufen Chang} 

\dedicatory{}

%
\thanks{Research supported by
Foundation of Natural Sciences of China (No. 11171091, 11371018).}
\address{
School of Mathematics and Information Sciences
\endgraf Henan Normal University \endgraf Xinxiang 453007, Henan
\endgraf P.R. China}
\email{xxl$@$henannu.edu.cn}

\address{
School of Mathematics and Information Sciences
\endgraf Henan Normal University \endgraf Xinxiang 453007, Henan
\endgraf P.R. China} %
\email{changxfff$@$163.com}



\begin{abstract}
In this paper, we study complete space-like $\lambda$-hypersurfaces in the Lorentzian space $\bbr^{n+1}_1$. As the result, we prove some rigidity theorems for these hypersurfaces including the complete space-like self-shrinkers in $\bbr^{n+1}_1$.
\end{abstract}

\maketitle

\section{Introduction} 

For $\veps=\pm 1$, let $\mathbb{E}^{n+1}_\veps$ be the Euclidean space $\mathbb{R}^{n+1}$ (when $\veps=1$) or the Lorentzian space $\mathbb{R}^{n+1}_{1}$ (when $\veps=-1$). The standard inner product on $\mathbb{E}^{n+1}_\veps$ is given by:
\begin{align*}
\langle X, Y\rangle=X_{1}Y_{1}+X_{2}Y_{2}+\cdots+\veps X_{n+1}Y_{n+1}.
\end{align*}

Let $M$ be an immersed hypersurface in $\mathbb{E}^{n+1}_\veps$ and $x: M^n\rightarrow\mathbb{E}^{n+1}_\veps$ be the corresponding immersion of $M$. In this paper, we also use $x$ to denote the position vector of $M$. Thus $x$, as well as the unit normal vector $N$ and the mean curvature $H$, are taken as smooth $\bbr^{n+1}$-valued functions on $M^n$. For a suitably chosen function $s$ on $M^n$, if the position vector $x$ of $M$ satisfies
\begin{equation}\label{eq1.1}
\begin{aligned}
H-s\langle x, N\rangle:=\lambda=\const,
\end{aligned}
\end{equation}
then $M$ is called a $\lambda$-hypersurface with the weight function $s$.

When $s\equiv 0$, the corresponding $\lambda$-hypersurfaces reduce to hypersurfaces with constant mean curvature which have been studied extensively. For example, Calabi considered in \cite{c} the maximum space-like hypersurfaces $M^n$ in the Lorentzian space $\bbr^{n+1}_1$ and proposed some Bernstein-type problems for a nonlinear equation; For a given complete space-like hypersurface $M$ in $\bbr^{n+1}_1$, it was proved by Xin (\cite{X}) that if the image of the Gauss map is inside a bounded subdomain of the hyperbolic $n$-space $\mathbb{H}^n$, then $M$ must be a hyperplane. A similar result was also proved earlier in (\cite{B.P}) with extra assumptions. In \cite{C-S-Z}, Cao, Shen and Zhu further extended the result by showing that if the image of the Gauss map lies inside a horoball of $\mathbb{H}^n$, $M$ is necessarily a hyperplane. Later, Wu (\cite{W}) generalized the above mentioned results and proved a more general Bernstein theorem for complete space-like hypersurfaces in Lorentzian space with constant mean curvature.

If $s$ is chosen to be constant and $\lambda=0$, $M$ is called a self-shrinker. It is known that self-shrinkers play an important role in the study of the mean curvature flow because they describe all possible blow ups at a given singularity of the mean curvature flow (\cite{C-W2}). There are also other mathematicians who have been studying the geometries of self-shrinkers and obtained a lot of interesting theorems, including some gap theorems and rigidity theorems for complete self-shrinkers. Details of this can been found in, for example, \cite{H.D.C-H.L}, \cite{C-W1}, \cite{Q.M.C-Y.J}, \cite{Q.D-Y.L.X}, \cite{Q.D-Y.L.X-L.Y}, \cite{Q-G-1}, \cite{H.L-Y.W} etc.

According to \cite{Q-G}, the concept of $\lambda$-hypersurfaces in the Euclidean space $\bbr^{n+1}$ were firstly introduced by M. Mcgonagle and J. Ross with $s=\frac{1}{2}$ (\cite{M-J});  Q. Guang (\cite{Q-G}) also defined the $\lambda$-hypersurfaces in $\bbr^{n+1}$ with $s=\frac{1}{2}$ and proved a Bernstein-type theorem showing that smooth $\lambda$-hypersuafaces which are entire graphs and with a polynomial volume growth are necessarily hyperplanes in $\bbr^{n+1}$.

If one takes $s(X)=-1$ in \eqref{eq1.1}, the corresponding $\lambda$-hypersurfaces are exactly what Q.-M. Cheng and G. Wei defined and studied in \cite{C-W2}, where the authors have successfully introduced a weighted volume functional and proved that the $\lambda$-hypersuafaces in the Euclidean space $\bbr^{n+1}$ are nothing but the critical points of the above functional. Later, Cheng, Ogaza and Wei (\cite{C-O-W}, \cite{C-W3}) have obtained some rigidity and Bernstein-type theorems for these complete $\lambda$-hypersurface. In particular, the following result is proved:

\begin{thm}[\cite{C-O-W}]\label{thm 1.1} Let $x: M^n\rightarrow \mathbb{R}^{n+1}$ be an $n$-dimensional complete $\lambda$-hypersurface with weight $s=-1$ and a polynomial area growth. Then, either $x$ is isometric to one of the following embedded hypersurfaces:
\begin{enumerate}
\item the sphere $S^{n}(r)\subset \mathbb{R}^{n+1}$ with radius $r>0$;
\item the hyperplane $\mathbb{R}^{n}\subset \mathbb{R}^{n+1}$;
\item the cylinder $S^{1}(r)\times\mathbb{R}^{n-1}\subset \mathbb{R}^{n+1}$;
\item the cylinder $S^{n-1}(r)\times\mathbb{R}\subset \mathbb{R}^{n+1}$,
\end{enumerate}
or, there exists some $p\in M^n$ such that the squared norm $S$ of the second fundamental form of $x$ satisfies
\begin{align*}
\left(\sqrt{S(p)-\frac{H^{2}(p)}{n}}+|\lambda|\frac{n-2}{2\sqrt{n(n-1)}}\right)^{2} +\frac{1}{n}(H(p)-\lambda)^{2}>1+\frac{n\lambda^{2}}{4(n-1)}.
\end{align*}
\end{thm}

In this paper, we consider space-like $\lambda$-hypersurfaces $x:M^n\to\bbr^{n+1}_1$ in the Lorentzian space $\bbr^{n+1}_1$. We first extend the definitions of $\lambda$-hypersurfaces and self-shrinkers to those in $\bbr^{n+1}_1$, and then generalize the $\mathcal{L}$-operator that has been effectively used by many authors (see the operators $\mathcal{\td L}$ and $\mathcal{L}$ defined, respectively in \eqref{2.4} and \eqref{2.14}). We shall be using these generalized operators to extend Theorem \ref{thm 1.1} to the
complete space-like $\lambda$-hypersurfaces in  $\bbr^{n+1}_1$.

Let $a$ be a nonzero constant and denote $\eps=\sgn(a\langle x, x\rangle)$ where $\lagl\cdot,\cdot\ragl$ is the Lorentzian product. We shall study $\lambda$-hypersurfaces in $\bbr^{n+1}_1$ either with weight $s=\eps a$ or with weight $s=\langle x, x\rangle$.
For a given hypersurface $M$, we always use $S$ to denote the squared norm of the second fundamental form, and use $A$ and $I$ to denote the shape operator and the identity map, respectively. Then the rigidity theorems we have proved in this paper are stated as follows:

\begin{thm}\label{thm 1.2} Let $x: M^n\rightarrow\mathbb{R}_{1}^{n+1}$ be a complete space-like $\lambda$-hypersurface with $s=\eps a$. Suppose that $\langle x, x\rangle$ does not change sign and
\begin{equation}\label{eq1.2}
\int_{M^n}\left(\left|\nabla\left(S-\frac{H^{2}}{n}\right)\right| +\left|\mathcal{L}\left(S-\frac{H^{2}}{n}\right)\right| \right)e^{-\frac{\eps a\langle x, x\rangle}{2}}dV_{M^n}<+\infty,
\end{equation}
where the differential operator $\mathcal{L}$ is defined by \eqref{2.4}. Then, either $x$ is totally umbilical and thus isometric to one of the following two hypersurfaces:
\begin{enumerate}
\item the hyperbolic space $\mathbb{H}^{n}(c)\subset\mathbb{R}_{1}^{n+1}$ with an arbitrary sectional curvature $c<0$;
\item the Euclidean space $\mathbb{R}^{n}\subset \mathbb{R}_{1}^{n+1}$,
\end{enumerate}
or, there exists some $p\in M^n$ such that, at $p$
\begin{equation}\label{1.3}
\left(\sqrt{S -\frac{H^{2} }{n}}-|\lambda|\frac{n-2}{2\sqrt{n(n-1)}}\right)^{2}+\frac{1}{n}(H -\lambda)^{2}-\frac{n\lambda^{2}}{4(n-1)}+\eps a<0.
\end{equation}
\end{thm}

\begin{thm}\label{thm 1.3} Let $x: M^n\rightarrow \mathbb{R}^{n+1}_{1}$ be a complete space-like $\lambda$-hypersurface with $s=\langle x, x\rangle$. Suppose that
\begin{align}
\int_{M^n}&\left(\left|\nabla \left(S-\frac{H^{2}}{n}\right)\right|+\left|\mathcal{\td L} \left(S-\frac{H^{2}}{n}\right)\right|\right)e^{-\frac{\langle x, x\rangle ^{2}}{4}}dV_{M^n}<+\infty,\label{eq1.4-1}
\\
&4A^{2}-\frac{4HA}{n}+\left(S-\frac{H^{2}}{n}\right)I\geq 0.\label{eq1.4-2}
\end{align}
where the differential operator $\mathcal{\td L}$ is defined by \eqref{2.14}. Then, either $x$ is totally umbilical and thus isometric to one of the following two embedded hypersurfaces:
\begin{enumerate}
\item the hyperbolic space $\mathbb{H}^{n}(c)\subset\mathbb{R}^{n+1}_{1}$ with an arbitrary $c<0$;
\item the Euclidean space $\mathbb{R}^{n}\subset\mathbb{R}^{n+1}_{1}$;
\end{enumerate}
or, there exists some $p\in M^n$ such that
\be\label{1.6}
\left(\sqrt{S(p)-\frac{H^{2}(p)}{n}} -|\lambda|\frac{n-2}{2\sqrt{n(n-1)}}\right)^{2}+\frac{1}{n}(H(p)-\lambda)^{2} -\frac{n\lambda^{2}}{4(n-1)}<0.
\ee
\end{thm}

Theorems \ref{thm 1.2} and \ref{thm 1.3} will be proved in Section $3$. Before doing this some necessary lemmas are given in Section $2$.

As a direct corollariy of Theorem \ref{thm 1.3}, we obtain

\begin{thm}\label{thm 1.4} Let $x: M^n\rightarrow \mathbb{R}^{n+1}_{1}$ be a complete space-like $\lambda$-hypersurface with $s=\langle x, x\rangle$. Suppose that
\eqref{eq1.4-1} and \eqref{eq1.4-2} are satisfied. If
\be\label{eq1.7}
\left(\sqrt{S-\fr1n H^{2}} -|\lambda|\frac{n-2}{2\sqrt{n(n-1)}}\right)^{2}+\frac{1}{n}(H-\lambda)^{2} -\frac{n\lambda^{2}}{4(n-1)}\geq 0,
\ee
then one of the following two conclusions must hold:
\begin{enumerate}
\item $\lambda\leq \left(\fr n2\right)^{\fr34}$, and $x$ is isometric to the hyperbolic space $\mathbb{H}^{n}(-r^{-2})\subset\mathbb{R}^{n+1}_{1}$ with $r\geq\left(\fr n2\right)^{\fr14}$;
\item $\lambda=0$ and $x$ is isometric to the Euclidean space $\mathbb{R}^{n}\subset\mathbb{R}^{n+1}_{1}$.
\end{enumerate}
\end{thm}

\proof If the condition \eqref{eq1.7} is satisfied for a hypersurface $\mathbb{H}^{n}(-r^{-2})$ with $r>0$, then by the fact that $x=rN$ we have
$$\lambda=H-\lagl x,x\ragl\lagl x,N\ragl=\fr nr-r^3.$$
If follows that
\begin{align}
0\leq&\left(\sqrt{S-\fr1n H^{2}} -|\lambda|\frac{n-2}{2\sqrt{n(n-1)}}\right)^{2}+\frac{1}{n}(H-\lambda)^{2} -\frac{n\lambda^{2}}{4(n-1)}\nnm\\
=&\fr1n(H^2-2\lambda H)=\fr1{r^2}(2r^4-n). \label{1.8}
\end{align}
Therefore $r\geq \left(\fr n2\right)^{\fr14}$ which implies directly that $\lambda\leq \left(\fr n2\right)^{\fr34}$. As for the Euclidean space $\bbr^n$, $\lambda=0$ is direct by the definition of $\lambda$-hypersurfaces.\endproof

A similar corollary of Theorem \ref{thm 1.2} can also be derived, which is omitted here.

\begin{cor}\label{cor 1.5} Let $x: M^n\rightarrow \mathbb{R}^{n+1}_{1}$ be a complete space-like $\lambda$-hypersurface with $s=\langle x, x\rangle$. Suppose $S-\frac{H^{2}}{n}$ is constant. If \eqref{eq1.4-2} and \eqref{eq1.7} are satisfied,
then $x$ is isometric to either the hyperbolic space $\mathbb{H}^{n}(-r^{-2})\subset\mathbb{R}^{n+1}_{1}$ with $r\geq\left(\fr n2\right)^{\fr14}$ or the hyperplane $\mathbb{R}^{n}\subset\mathbb{R}^{n+1}_{1}$.
\end{cor}

\proof Since $S-\frac{H^{2}}{n}$ is constant, the condition \eqref{eq1.4-1} in Theorem \ref{thm 1.3} is trivially satisfied. Then Corollary \ref{cor 1.5} follows direct from Theorem \ref{thm 1.4}.

{\rmk\rm For the special case that $\lambda=0$, that is, for the ``{\em self-shrinker}'' case, the following two conclusions can be easily seen from Theorem \ref{thm 1.4}:

\begin{thm}\label{thm 1.6} Let $x: M^n\rightarrow \mathbb{R}^{n+1}_{1}$ be a complete space-like hypersurface self-shrinker with $s=\langle x, x\rangle$. Suppose that \eqref{eq1.4-1} and \eqref{eq1.4-2} are satisfied, then
$x$ is isometric to one of the following two embedded hypersurfaces:
\begin{enumerate}
\item the hyperbolic space $\mathbb{H}^{n}\left(-\fr1{\sqrt{n}}\right)\subset\mathbb{R}^{n+1}_{1}$;
\item the Euclidean space $\mathbb{R}^{n}\subset\mathbb{R}^{n+1}_{1}$.
\end{enumerate}
\end{thm}

\proof When $\lambda=0$, it is clear that \eqref{eq1.7} is trivially satisfied. Furthermore, for a hyperbolic space $\mathbb{H}^{n}(-r^{-2})\subset\mathbb{R}^{n+1}_{1}$, $\lambda=0$ also implies that $r^2=\sqrt{n}$.
\endproof

\begin{cor} Let $x: M^n\rightarrow \mathbb{R}^{n+1}_{1}$ be a complete space-like self-shrinker with $s=\langle x, x\rangle$. If $S-\frac{H^{2}}{n}$ is constant and
\eqref{eq1.4-2} is satisfied, then $x$ is isometric to the either the hyperbolic space $\mathbb{H}^{n}\left(-\fr1{\sqrt{n}}\right)\subset\mathbb{R}^{n+1}_{1}$ or the hyperplane $\mathbb{R}^{n}\subset\mathbb{R}^{n+1}_{1}$.
\end{cor}

\proof  The assumption that $S-\frac{H^{2}}{n}$ is constant directly means that \eqref{eq1.4-1} is trivially satisfied. \endproof

\section{Preliminaries and necessary lemmas}

Firstly we fix the following convention for the ranges of indices:
$$1\leq i,j,k,\cdots \leq n,\quad 1\leq A,B,C\cdots\leq n+1.$$

Let $x: M^n\rightarrow\mathbb{R}_{1}^{n+1}$ be a connected space-like hypersurface of the $(n+1)$-dimensional Lorentzian space $\mathbb{R}_{1}^{n+1}$ and $\{e_{A}\}_{A=1}^{n+1}$ be a local orthonormal frame field of $\mathbb{R}_{1}^{n+1}$ along $x$ with dual coframe field $\{\omega^{A}\}_{A=1}^{n+1}$ such that, when restricted to $x$, $e_{1}, \ldots, e_{n}$ are tangent to $x$ and thus $N:=e_{n+1}$ is the unit normal vector of $x$. Then with the connection forms $\omega^B_A$ we have
$$dx=\sum_{i}\omega^{i}e_{i},\quad de_{i}=\sum_{j}\omega^{j}_{i}e_{j}+\omega_{i}^{n+1}e_{n+1}, \quad de_{n+1}=\sum_{i}\omega_{n+1}^{i}e_{i}.$$
By restricting these forms to $M^{n}$ and using Cartan's lemma, we have
$$\omega^{n+1}=0,\quad \omega_{i}^{n+1}=\omega^{i}_{n+1}=\sum_{j=1}^{n}h_{ij}\omega^{j}, \quad h_{ij}=h_{ji}, $$
where $h_{ij}$ are nothing but the components of the second fundamental form $h$ of $x$, that is, $h=\sum h_{ij}\omega^i\omega^j$. Then the mean curvature $H$ of $x$ is given by $H=\sum_{j=1}^{n}h_{jj}$. Denote
\be
h_{ijk}=(\nabla h)_{ijk}=(\nabla _{k}h)_{ij},\quad h_{ijkl}=(\nabla^2h)_{ijkl}=(\nabla_l(\nabla h))_{ijk}
\ee
where $\nabla$ is the Levi-Civita connection of the induced metric and $\nabla_i:=\nabla_{e_i}$. Then the Gauss equations, Codazzi equations and Ricci identities are given respectively by
\begin{align}
R_{ijkl}=&h_{ik}h_{jl}-h_{il}h_{jk},
\quad
h_{ijk}=h_{ikj},\label{eq2.1}
\\
h_{ijkl}&-h_{ijlk}=\sum_{m=1}^{n}h_{im}R_{jmkl}+\sum^{n}_{m=1}h_{mj}R_{imkl},\label{eq2.3}
\end{align}
where $R_{ijkl}$ are the components of the Riemannian curvature tensor. For a function $F$ defined on $M$, the covariant derivatives of $F$ are denoted by
$$F_{,i}=(\nabla F)_i=\nabla_{i}F,\quad F_{,ij}=(\nabla^2 F)_{ij}=(\nabla_{j}(\nabla F))_{i},\quad \cdots.$$
Let $\Delta$ be the Laplacian operator of the induced metric on $M^n$. In case that $\lagl x,x\ragl$ does not change its sign, we can define
\begin{equation}\label{2.4}
\mathcal{L}v=\triangle v-\eps a \langle x, \nabla v\rangle,\quad \forall v\in C^2(M^n),
\end{equation}
where, for any constant $a$, $\epsilon=\sgn(a\langle x, x\rangle)$. Then $\mathcal{L}$ is an elliptic operator and
\be\label{ad2.1}
\mathcal{L}v=e^{\frac{\eps a\langle x, x\rangle}{2}}\dv \left(e^{-\frac{\eps a\langle x, x\rangle}{2}}\nabla v\right),\quad \forall v\in C^2(M^n).
\ee

In fact, for $v\in C^2(M^n)$, we find
\begin{align*}
&e^{\frac{\eps a\langle x, x\rangle}{2}}\dv \left(e^{-\frac{\eps a\langle x, x\rangle}{2}}\nabla v\right)\\
=&e^{\frac{\eps a\langle x, x\rangle}{2}}\left(e^{-\frac{\eps a\langle x, x\rangle}{2}}\dv (\nabla v)+\left\langle \nabla e^{-\frac{\eps a\langle x, x\rangle}{2}}, \nabla v\right\rangle\right)\\
=&e^{\frac{\eps a\langle x, x\rangle}{2}}\left(e^{-\frac{\eps a\langle x, x\rangle}{2}}\triangle v+e^{-\frac{\eps a\langle x, x\rangle}{2}}(-\epsilon a\langle x, x_{i}\rangle)\langle e_{i}, \nabla v\rangle\right)\\
=&\triangle v-\epsilon a\langle x, \nabla v \rangle=\mathcal{L}v.\\
\end{align*}

\begin{lem}[cf. \cite{T-W}]\label{lem2.1} Let $x:M^n\to \mathbb{R}_{1}^{n+1}$ be a complete space-like hypersurface for which $\lagl x,x\ragl$ does not change its sign. Then, for any $C^{1}$-function $u$ on $M^n$ with compact support, it holds that
\begin{equation}\label{eq2.5}
\int_{M^n}u(\mathcal{L}v)e^{-\frac{\eps a\langle x, x\rangle}{2}}dV_{M^n}=-\int_{M^n}\langle \nabla v, \nabla u\rangle e^{-\frac{\eps a\langle x, x\rangle}{2}}dV_{M^n},\quad\forall v\in C^2(M^n).
\end{equation}
\end{lem}

\begin{proof} By \eqref{ad2.1} we find
\begin{align*}
\int_{M^n}&u(\mathcal{L}v)e^{-\frac{\eps a\langle x, x\rangle}{2}}dV_{M^n}=\int_{M^n}u\left(e^{\frac{\eps a\langle x, x\rangle}{2}}\dv \left(e^{-\frac{\eps a\langle x, x\rangle}{2}}\nabla v\right)\right)e^{-\frac{\eps a\langle x, x\rangle}{2}}dV_{M^n}\\
=&\int_{M^n}u\,\dv \left(e^{-\frac{\eps a\langle x, x\rangle}{2}}\nabla v\right)dV_{M^n}\\
=&\int_{M^n}\left(\dv \left(u e^{-\frac{\eps a\langle x, x\rangle}{2}}\nabla v\right)-\left\langle \nabla u, e^{-\frac{\eps a\langle x, x\rangle}{2}}\nabla v\right\rangle\right)dV_{M^n}\\
=&\int_{M^n} \dv \left(u e^{-\frac{\eps a\langle x, x\rangle}{2}}\nabla v\right)dV_{M^n}-\int_{M^n}\langle \nabla u, \nabla v\rangle e^{-\frac{\eps a\langle x, x\rangle}{2}}dV_{M^n}.
\end{align*}
Hence there are two cases to be considered:

Case (1): $M^n$ is compact without boundary. In this case, we can directly use the divergence theorem to get
$$\int_{M^n}\dv \left(u e^{-\frac{\eps a\langle x, x\rangle}{2}}\nabla v\right)dV_{M^n}=0.$$

Case (2): $M^n$ is complete and noncompact. In this case, we can find a geodesic ball $B_{r}(o)$ big enough such that $\supp u\subset B_{r}(o)$. It follows that
\begin{align*}
\int_{M^n}&\dv \left(u e^{-\frac{\eps a\langle x, x\rangle}{2}}\nabla v\right)dV_{M^n}=\int_{B_{r}(o)}\dv \left(u e^{-\frac{\eps a\langle x, x\rangle}{2}}\nabla v\right)dV_{B_{r}(o)}\\
=&-\int_{\partial B_{r}(o)}\left\langle N, u e^{-\frac{\eps a\langle x, x\rangle}{2}}\nabla v \right \rangle dV_{\partial B_{r}(o)}
=0.
\end{align*}
It follows that
$$\int_{M^n}u(\mathcal{L}v)e^{-\frac{\eps a\langle x, x\rangle}{2}}dV_{M^n}=-\int_{M^n}\langle \nabla v, \nabla u\rangle e^{-\frac{\eps a\langle x, x\rangle}{2}}dV_{M^n}.$$
\end{proof}

\begin{cor}\label{cor2.2} Let $x: M^n\rightarrow\mathbb{R}_{1}^{n+1}$ be a complete space-like hypersurface. If $u,v $ are $C^2$-functions satisfying
\begin{equation}\label{eq2.6}
\int_{M^n}(|u\nabla v|+|\nabla u||\nabla v|+|u\mathcal{L} v|)e^{-\frac{\eps a\langle x, x\rangle}{2}} dV_{M^n}<+\infty,
\end{equation}
then we have
\begin{equation}
\int_{M^n}u(\mathcal{L}v)e^{-\frac{\eps a\langle x, x\rangle}{2}}dV_{M^n}=-\int_{M^n}\langle\nabla u, \nabla v\rangle e^{-\frac{\eps a\langle x, x\rangle}{2}}dV_{M^n}.
\end{equation}

\end{cor}
\begin{proof} Within this proof, we will use square brackets $[\cdot]$ to denote weighted integrals
\begin{equation}
[f]=\int_{M^n}fe^{-\frac{\eps a\langle x, x\rangle}{2}}dV_{M^n}.
\end{equation}
 Given any $\phi$ that is $C^1$-with compact support, we can apply Lemma \ref{lem2.1} to $\phi u$ and $v$ to get
\begin{equation}\label{eq2.9}
[\phi u\mathcal{L}v]=-[\phi \langle \nabla v, \nabla u\rangle]-[u\langle \nabla v, \nabla \phi \rangle].
\end{equation}
Now we fix one point $o\in M$ and, for each $j=1,2,\cdots$, let $B_j$ be the intrinsic ball of radius $j$ in $M^n$ centered at $o$. Define $\phi_{j}$ to be one smooth cutting-off function on $M^n$ that cuts off linearly from one to zero between $B_{j}$ and $B_{j+1}$. Since $|\phi_{j}|$ and $|\nabla \phi_{j}|$ are bounded by one, $\phi_{j}\rightarrow 1$ and $|\nabla \phi_{j}|\rightarrow 0$, as $j\to+\infty$. Then the dominated convergence theorem (which applies because of $\eqref{eq2.6}$) shows that, as $j\to+\infty$, we have the following limits:
\begin{equation}
[\phi_{j} u\mathcal{L}v]\rightarrow[u\mathcal{L}v],
\end{equation}
\begin{equation}
[\phi_{j}\langle \nabla v, \nabla u\rangle]\rightarrow [\langle \nabla v, \nabla u\rangle],
\end{equation}
\begin{equation}
[u\langle \nabla v, \nabla \phi \rangle] \rightarrow 0.
\end{equation}
Replacing $\phi$ in \eqref{eq2.9} with $\phi_j$, we obtain the corollary.
\end{proof}

Next we consider the case that $s=\langle x, x\rangle$ and define
\begin{equation}\label{2.14}
\mathcal{\td L}v=\triangle v-\langle x, x\rangle \langle x, \nabla v\rangle,\quad \forall v\in C^2(M^n).
\end{equation}
Then, similar to \eqref{ad2.1}, we have for all $v\in C^2(M^n)$,
\begin{align}
&e^{\frac{\langle x, x\rangle^{2}}{4}}\dv \left(e^{-\frac{\langle x, x\rangle^{2}}{4}}\nabla v\right)\nnm\\
=&e^{\frac{\langle x, x\rangle^{2}}{4}}\left(e^{-\frac{\langle x, x\rangle^{2}}{4}}\dv (\nabla v)+\left\langle \nabla e^{-\frac{\langle x, x\rangle^{2}}{4}}, \nabla v\right\rangle\right)\nnm\\
=&e^{\frac{\langle x, x\rangle^{2}}{4}}\left(e^{-\frac{\langle x, x\rangle^{2}}{4}}\triangle v+e^{-\frac{\langle x, x\rangle^{2}}{4}}\left(- \frac{2\langle x, x\rangle}{4}2\langle x, x_{i}\rangle\right)\langle e_{i}, \nabla v\rangle\right)\nnm\\
=&\triangle v- \langle x, x\rangle\langle x, \nabla v \rangle=\mathcal{\td L}v.\label{ad2.2}
\end{align}

\begin{lem}\label{lem2.3} If $x:M^n\to \mathbb{R}^{n+1}_{1}$ is a complete space-like hypersurface, $u$ is a $C^1$-function with compact support, and $v$ is a $C^2$-function, then
\begin{equation}
\int_{M^n}u(\mathcal{\td L}v)e^{-\frac{\langle x, x\rangle^{2}}{4}}dV_{M^n}=-\int_{M^n}\langle \nabla v, \nabla u\rangle e^{-\frac{\langle x, x\rangle^{2}}{4}}dV_{M^n}.
\end{equation}
\end{lem}

\begin{proof} Using \eqref{ad2.2} we have
\begin{align*}
\int_{M^n}&u(\mathcal{\td L}v)e^{-\frac{\langle x, x\rangle^{2}}{4}}dV_{M^n}=\int_{M^n}u\left(e^{\frac{\langle x, x\rangle^{2}}{4}}\dv \left(e^{-\frac{\langle x, x\rangle^{2}}{4}}\nabla v\right)\right)e^{-\frac{\langle x, x\rangle^{2}}{4}}dV_{M^n}\\
=&\int_{M^n}u\,\dv \left(e^{-\frac{\langle x, x\rangle^{2}}{4}}\nabla v\right)dV_{M^n}\\
=&\int_{M^n}\left(\dv \left(u e^{-\frac{\langle x, x\rangle^{2}}{4}}\nabla v\right)-\left\langle \nabla u, e^{-\frac{\langle x, x\rangle^{2}}{4}}\nabla v\right\rangle\right)dV_{M^n}\\
=&\int_{M^n}\dv \left(u e^{-\frac{\langle x, x\rangle^{2}}{4}}\nabla v\right)dV_{M^n}-\int_{M^n}\langle \nabla u, \nabla v\rangle e^{-\frac{\langle x, x\rangle^{2}}{4}}dV_{M^n}.
\end{align*}

 (1) If $M^n$ is compact without boundary, then by the divergence theorem,
 $$\int_{M^n}\dv \left(u e^{-\frac{\langle x, x\rangle^{2}}{4}}\nabla v\right)dV_{M^n}=0.$$

 (2) If $M^n$ is complete and noncompact, then there exists some geodesic ball $B_{r}(o)$ big enough such that $\supp u \subset B_{r}(o)$. It follows that
\begin{align*}
&\int_{M^n}\dv \left(u e^{-\frac{\langle x, x\rangle^{2}}{4}}\nabla v\right)dV_{M^n}=\int_{B_{r}(o)}\dv \left(u e^{-\frac{\langle x, x\rangle^{2}}{4}}\nabla v\right)dV_{B_{r}(o)}\\
=&-\int_{\partial B_{r}(o)}\left\langle N, u e^{-\frac{\langle x, x\rangle^{2}}{4}}\nabla v\right\rangle dV_{\partial B_{r}(o)}
=0.
\end{align*}
Therefore
$$\int_{M^n}u(\mathcal{\td L}v)e^{-\frac{\langle x, x\rangle^{2}}{4}}dV_{M^n}=-\int_{M^n}\langle \nabla v, \nabla u\rangle e^{-\frac{\langle x, x\rangle^{2}}{4}}dV_{M^n}.$$
\end{proof}

\begin{cor}\label{cor2.4} Let $x: M^n\rightarrow \mathbb{R}^{n+1}_{1}$ be a complete space-like hypersurface. If $u,v$ are $C^2$-functions satisfying
\begin{equation}\label{eq2.15}
\int_{M^n}(|u\nabla v|+|\nabla u||\nabla v|+|u\mathcal{\td L}v|)e^{-\frac{\langle x, x\rangle^{2}}{4}}dV_{M^n}<+\infty,
\end{equation}
then we have
\begin{equation}
\int_{M^n}u(\mathcal{\td L}v)e^{-\frac{\langle x, x\rangle^{2}}{4}}dV_{M^n}=-\int_{M^n}\langle \nabla u, \nabla v\rangle e^{-\frac{\langle x, x\rangle^{2}}{4}}dV_{M^n}.
\end{equation}
\end{cor}
\proof The proof here is same as that of Corollary \ref{cor2.2} and is omitted.\endproof

The following lemma is also needed in this paper:

\begin{lem}[\cite{O}]\label{lem2.5}
Let $\mu_{1}, \cdots, \mu_{n}$ be real numbers satisfying
\begin{align*}
\sum_{i}\mu_{i}=0,\quad\sum_{i}\mu_{i}^{2}=\beta^{2},
\end{align*}
with $\beta$ a nonnegative constant. Then
\begin{align*}
-\frac{n-2}{\sqrt{n(n-1)}}\beta^{3}\leq\sum_{i} \mu_{i}^{3}\leq\frac{n-2}{\sqrt{n(n-1)}}\beta^{3}
\end{align*}
with either equality holds if and only if $(n-1)$ of $\mu_{i}$ are equal to each other.
\end{lem}

\section{Proof of the main theorems}In this section, we give the proofs of our main theorems.

(1) \textsl{Proof of Theorem 1.2}

Since $H-\eps a\langle x, N\rangle=\lambda$, we have
$$\begin{aligned}
H_{,i}=(\lambda+\eps a\langle x, N\rangle),_{i}=\eps a\langle x, N\rangle_{,i}=\sum_{k}\eps a h_{ik}\langle x, e_{k}\rangle
\end{aligned},$$
$$\begin{aligned}
H_{ij}=&\sum_{k}\eps a h_{ikj}\langle x, e_{k}\rangle +\sum_{k}\eps a h_{ik}\langle X_{j}, e_{k}\rangle +\sum_{k}\eps a h_{ik}\langle x, e_{k,j}\rangle\\
=&\sum_{k}\eps a h_{ikj}\langle x, e_{k}\rangle+\eps a h_{ij}+\sum_{k}\eps a h_{ik}h_{kj}\langle x, N\rangle \\
=&\sum_{k}\eps a h_{ikj}\langle x, e_{k} \rangle+\eps a h_{ij}+\sum_{k} h_{ik}h_{kj}(H-\lambda).
\end{aligned}$$
Using the Codazzi equation in \eqref{eq2.1} we infer
$$\triangle H=\sum_{i}H_{,ii}=\eps a \langle x, \nabla H\rangle+\eps a H+S(H-\lambda),$$
where $S=\sum_{i,k}h_{ik}^{2}$.
It then follows that
\begin{align*}
\mathcal{L}H=\triangle H-\eps a\langle x, \nabla H\rangle=a
\epsilon H+S(H-\lambda),
\end{align*}
implying that
\begin{align}
\frac{1}{2}\mathcal{L}H^{2}=&\frac{1}{2}(\triangle H^{2}-\eps a \langle x, \nabla H^{2}\rangle)\nnm\\
=&\frac{1}{2}\left(\sum_{i}(H^{2})_{,ii}-\eps a \langle x, \nabla H^{2}\rangle\right)\nnm\\
=&\frac{1}{2}(2|\nabla H|^{2}+2H\triangle H-2\eps a H \langle x, \nabla H\rangle)\nnm\\
=&|\nabla H|^{2}+H(\triangle H-\eps a \langle x, \nabla H\rangle)\nnm\\
=&|\nabla H|^{2}+\eps a H^{2} +SH(H-\lambda).\label{ad3.9}
\end{align}
By making use of the Ricci identities and the Gauss-Codazzi equations, we have
\begin{align*}
\mathcal{L}h_{ij}=&\triangle h_{ij}-\eps a \langle x, \nabla h_{ij}\rangle
=\sum_{k}h_{ki,jk}-\eps a \langle x, \nabla h_{ij}\rangle\\
=&\sum_{k}h_{ki,kj}+\sum_{m,k}h_{mi}R^{m}_{kjk}+\sum_{k,m}h_{km}R^{m}_{ijk}-\eps a \langle x, \nabla h_{ij}\rangle\\
=&\sum_{k}h_{kk,j}+\sum_{m,k}h_{mi}R_{kmjk}+\sum_{k,m}h_{km}R_{imjk}-\eps a \langle x, \nabla h_{ij}\rangle\\
=&H_{,ij}-H\sum_{m}h_{im}h_{mj}+Sh_{ij}-\eps a \langle x, \nabla h_{ij}\rangle\\
=&(\eps a +S)h_{ij}-\lambda \sum_{k}h_{ik}h_{kj} \vspace{0.2cm}.
\end{align*}
Therefore, it holds that
\begin{align}
\frac{1}{2}\mathcal{L}S=&\frac{1}{2}\left(\triangle \sum_{i,j}(h_{ij})^{2}-\sum_{k}\eps a \langle x, e_{k}\rangle\left(\sum_{i,j}(h_{ij})^{2}\right)_{,k}\right)\nnm\\
=&\sum_{i,j,k}h_{ijk}^{2}+(\eps a+S)S-\lambda \sum_{i,j,k}h_{ik}h_{kj}h_{ij}\nnm\\
=&\sum_{i,j,k}h_{ijk}^{2}+(\eps a+S)S-\lambda f_{3},\label{ad3.10}
\end{align}
where $$f_{3}=\sum_{i,j,k}h_{ij}h_{jk}h_{ki}.$$
Let $\lambda_i$ be the principal curvatures of $x$ and denote $$\mu_{i}=\lambda _{i}-\frac{H}{n},\quad 1\leq i\leq n.$$
For any point $p\in M^n$, suitably choosing $\{e_{1},e_{2},\cdots,e_{n}\}$ around $p$ such that $h_{ij}(p)=\lambda_{i}(p)\delta_{ij}$. Then, at the given point $p$,
$$f_{3}=\sum_{i}\lambda_{i}^{3}=\sum_{i}\left(\mu_{i} +\frac{H}{n}\right)^{3}=B_{3}+\frac{3}{n}HB+\frac{1}{n^{2}}H^{3},$$
where
\begin{align*}
B=\sum_{i}\mu_{i}^{2}=S-\frac{H^{2}}{n},\quad B_{3}=\sum_{i}\mu_{i}^{3}.
\end{align*}
By a direct computation with \eqref{ad3.9} and \eqref{ad3.10}, we have
\begin{align*}
\frac{1}{2}\mathcal{L}B=&\frac{1}{2}\mathcal{L}S -\frac{1}{n}\left(\frac{1}{2}\mathcal{L}H^{2}\right) \\
=&\sum_{i,j,k}h_{ijk}^{2}+(\eps a +S)S-\lambda f_{3}-\frac{1}{n}(|\nabla H|^{2} +\eps a H^{2}+SH(H-\lambda))\\
=&\sum_{i,j,k}h_{ijk}^{2}-\frac{1}{n}|\nabla H|^{2}+(\eps a +S)S-\lambda f_{3} -\frac{1}{n}\eps a H^{2}-S(H-\lambda)\frac{H}{n}\\
=&\sum_{i,j,k}h_{ijk}^{2}-\frac{1}{n}|\nabla H|^{2}+(\eps a+B)B +\frac{H^{2}B}{n}-\lambda B_{3}-\frac{2}{n}\lambda HB.
\end{align*}
Since
\begin{center}
$\sum_{i}\mu_{i}=0$, \quad $\sum_{i}\mu_{i}^{2}=B$,
\end{center}
we have by Lemma \ref{lem2.5}
$$|B_{3}|\leq\frac{n-2}{\sqrt{n(n-1)}}B^{\frac{3}{2}},$$
where the equality holds if and only if at least $n-1$ of $\mu_{i}$s are equal. Consequently,
\begin{align*}
\frac{1}{2}\mathcal{L}B\geq &\sum_{i,j,k}h_{ijk}^{2}-\frac{1}{n}|\nabla H|^{2}+(\eps a+B)B +\frac{1}{n}H^{2}B-|\lambda|\frac{n-2}{\sqrt{n(n-1)}}B^{\frac{3}{2}} -\frac{2}{n}\lambda HB \\
=&\sum_{i,j,k}h_{ijk}^{2}-\frac{1}{n}|\nabla H|^{2}+B\left((B+\eps a)+\frac{1}{n}H^{2} - |\lambda|\frac{n-2}{\sqrt{n(n-1)}}B^{\frac{1}{2}}-\frac{2}{n}\lambda H\right)\\
=&\sum_{i,j,k}h_{ijk}^{2}-\frac{1}{n}|\nabla H|^{2}+B\left(\left(\sqrt{B} -|\lambda|\frac{n-2}{2\sqrt{n(n-1)}}\right)^{2} + \frac{1}{n}(H-\lambda)^{2}+\eps a -\frac{n\lambda ^{2}}{4(n-1)}\right).
\end{align*}
Because of \eqref{eq1.2}, we can apply Corollary \ref{cor2.2} to functions $1$ and $B=S-\frac{H^{2}}{n}$ to obtain
\begin{align}
0\geq &\int_{M^n}\left(\sum_{i,j,k}h_{ijk}^{2}-\frac{1}{n}|\nabla H|^{2}\right)e^{-\epsilon\frac{a \langle x, x\rangle}{2}}dV_{M^n}\nnm\\
&+\int_{M^n}B\left(\left(\sqrt{S -\frac{H^{2} }{n}}-|\lambda|\frac{n-2}{2\sqrt{n(n-1)}}\right)^{2}+\frac{1}{n}(H -\lambda)^{2}-\frac{n\lambda^{2}}{4(n-1)}+\eps a\right)e^{-\epsilon\frac{a \langle x, x\rangle}{2}}dV_{M^n}.\label{ad3.1}
\end{align}
On the other hand, by use of the Codazzi equations and the Schwarz inequality, we find
\begin{align*}
\sum_{i,j,k}h_{ijk}^{2}=3\sum_{i\neq k}h_{iik}^{2}+\sum_{i}h_{iii}^{2}+\sum_{i\neq j\neq k\neq i}h_{ijk}^{2}, \quad \frac{1}{n}|\nabla H|^{2}\leq \sum_{i,k}h_{iik}^{2}.
\end{align*}
So that
\be\label{ad3.4}\sum_{i,j,k}h_{ijk}^{2}-\frac{1}{n}|\nabla H|^{2}\geq 2\sum_{i\neq k}h_{iik}^{2}+\sum_{i\neq j\neq k \neq i}h_{ijk}^{2}\geq 0,\ee
in which the equalities hold if and only if $h_{ijk}=0$ for any $i,j,k$.

If $B\not\equiv 0$ and, for all $p\in M^n$, \eqref{1.3} does not hold, that is
$$\left(\sqrt{S -\frac{H^{2} }{n}}-|\lambda|\frac{n-2}{2\sqrt{n(n-1)}}\right)^{2}+\frac{1}{n}(H -\lambda)^{2}-\frac{n\lambda^{2}}{4(n-1)}+\eps a \geq 0$$
everywhere on $M^n$, then the right hand side of \eqref{ad3.1} is nonnegative. It then follows that
\be\label{ad3.2}\sum_{i,j,k}h_{ijk}^{2}-\frac{1}{n}|\nabla H|^{2}\equiv 0,\ee
and
\be\label{ad3.3}
\left(\sqrt{S -\frac{H^{2} }{n}}-|\lambda|\frac{n-2}{2\sqrt{n(n-1)}}\right)^{2}+\frac{1}{n}(H -\lambda)^{2}-\frac{n\lambda^{2}}{4(n-1)}+\eps a\equiv 0
\ee
on where $B\neq 0$. By \eqref{ad3.4} and \eqref{ad3.2}, the second fundamental form $h$ of $x$ is parallel. In particular, $x$ is isoparametric and thus both $B$ and $H$ are constant.
Since $B\neq 0$, the equality \eqref{ad3.3} shows that $x$ is a complete isoparametric space-like hypersurface in $\bbr^{n+1}_1$ of exactly two distinct principal curvatures one of which is simple. It then follows by \cite{HBL} and $B\neq 0$ that $x$ is isometric to one of the product spaces $\mathbb{H}^{n-1}(c)\times \bbr^1\subset\bbr^{n+1}_1$ and $\mathbb{H}^1(c)\times \bbr^{n-1}\subset\bbr^{n+1}_1$. But it is clear that, for both of these two product spaces, the function $\lagl x,x\ragl$ does change its sign, contradicting the assumption. This contradiction proves that either $B\equiv 0$, namely, $x$ is totally umbilical and isometric to either of the hyperbolic $n$-space $\mathbb{H}^n(c)\subset\bbr^{n+1}_1$ and the Euclidean $n$-space $\bbr^n\subset\bbr^{n+1}_1$, or there exists some $p\in M^n$ such that \eqref{1.3} holds.

The proof of Theorem \ref{thm 1.2} is thus finished.\endproof

(2) \textsl{Proof of Theorem 1.3}:

Since the idea and method here are the same as those in the proof of Theorem
\ref{thm 1.2}, we omit the computation detail.

First, by $H-\langle x, x\rangle\langle x, N\rangle=\lambda$, we have
\begin{align*}
H_{,i}=&2\langle x, e_{i}\rangle\langle x, N\rangle+\langle x, x\rangle\sum_{k}h_{ik}\langle x, e_{k}\rangle,\\
H_{,ij}=&2\delta_{ij}\langle x, N\rangle+2h_{ij}\langle x, N\rangle^{2}+2\sum_{k}h_{jk}\langle x, e_{i}\rangle\langle x, e_{k}\rangle\\
&+2\sum_{k}h_{ik}\langle x, e_{k}\rangle\langle x, e_{j}\rangle+\sum_{k}h_{ikj}\langle x, x\rangle\langle x, e_{k}\rangle\\
&+\langle x, x\rangle h_{ij}+\sum_{k}h_{ik}h_{kj}(H-\lambda).
\end{align*}
Then by using the Codazzi equation in \eqref{eq2.1}, we find
\begin{align*}
\triangle H=&2n\langle x, N\rangle+2H\langle x, N\rangle^{2}+4\sum_{i,k}h_{ik}\langle x, e_{i}\rangle\langle x, e_{k}\rangle\\
&+\sum_{i}H_{,i}\langle x, x\rangle\langle x, e_{i}\rangle+H\langle x, x\rangle+S(H-\lambda).
\end{align*}

Secondly, by the definition of $\mathcal{\td L}$, we find
\begin{align*}
\mathcal{\td L}H=&\triangle H-\langle x, x\rangle\langle x, \nabla H\rangle\\
=&2n\langle x, N\rangle+2H\langle x, N\rangle^{2}+4\sum_{i,k}h_{ik}\langle\ X, e_{i}\rangle \langle x, e_{k}\rangle+H\langle x, x\rangle+S(H-\lambda),
\end{align*}
implying
\begin{align}
\frac{1}{2}\mathcal{\td L}H^{2}=&\frac{1}{2}(\triangle H^{2}-\langle x, x\rangle\langle x, \nabla H^{2}\rangle)\nnm\\
=&|\nabla H|^{2}+2nH\langle x, N\rangle+2H^{2}\langle x, N\rangle^{2}\nnm\\
&+4H\sum_{i,k}h_{ik}\langle x, e_{i}\rangle\langle x, e_{k}\rangle
+H^{2}\langle x, x\rangle+SH(H-\lambda).\label{ad3.7}
\end{align}
On the other hand
\begin{align*}
\mathcal{\td L}h_{ij}=&\triangle h_{ij}-\langle x, x\rangle \langle x, \nabla h_{ij}\rangle\\
=&H_{,ij}+\sum_{k,m}h_{mi}R_{kmjk}+\sum_{k,m}h_{km}R_{imjk}-\langle x, x\rangle\langle x, \nabla h_{ij}\rangle\\
=&H_{,ij}+Sh_{ij}-H\sum_{m}h_{mi}h_{mj}-\langle x, x\rangle\langle x, \nabla h_{ij}\rangle\\
=&2\delta_{ij}\langle x, N\rangle+2h_{ij}\langle x, N\rangle^{2}+2\sum_{k}h_{jk}\langle x, e_{i}\rangle\langle x, e_{k}\rangle\\
&+2\sum_{k}h_{ik}\langle x, e_{k}\rangle\langle x, e_{j}\rangle
+\sum_{k}h_{ikj}\langle x, x\rangle\langle x, e_{k}\rangle+\langle x, x\rangle h_{ij}\\
&+\sum_{k}h_{ik}h_{kj}(H-\lambda)+Sh_{ij}-H\sum_{m}h_{mi}h_{mj}-\langle x, x\rangle\langle x, \nabla h_{ij}\rangle\\
=&2\delta_{ij}\langle x, N\rangle+2h_{ij}\langle x, N\rangle^{2}+2\sum_{k}h_{jk}\langle x, e_{i}\rangle\langle x, e_{k}\rangle\\
&+2\sum_{k}h_{ik}\langle x, e_{k}\rangle\langle x, e_{j}\rangle
+\langle x, x\rangle h_{ij}+Sh_{ij}-\lambda \sum_{k}h_{ik}h_{kj}.
\end{align*}
It follows that
\begin{align}
\frac{1}{2}\mathcal{\td L}S=&\frac{1}{2}\left(\triangle \sum_{i,j}(h_{ij})^{2}-\langle x, x\rangle\left\langle x, \nabla \left(\sum_{i,j}(h_{ij})^{2}\right)\right\rangle\right)\nnm\\
=&\sum_{k}h_{ijk}^{2}+2H\langle x, N\rangle+2S\langle x, N\rangle^{2}+4\sum_{i,j,k}h_{ij}h_{jk}\langle x, e_{i}\rangle\langle x, e_{k}\rangle\nnm\\
&+\langle x, x\rangle S+S^{2}-\lambda f_{3},\label{ad3.8}
\end{align}
where again $f_{3}=\sum_{i,j,k}h_{ij}h_{jk}h_{ki}$.

Denote by $x^\top=\langle x, e_{i}\rangle e_i$ be the tangential part of the position vector $x$. Then, as in the proof of Theorem \ref{thm 1.2}, we can choose a suitable frame field $\{e_{1},e_{2},\cdots, e_{n}\} $ making diagonal the second fundamental form $h_{ij}$ around each point $p\in M^n$, and perform a direct computation using \eqref{eq1.4-2}, \eqref{ad3.7} and \eqref{ad3.8} to obtain
\begin{align*}
\frac{1}{2}\mathcal{\td L}B=&\frac{1}{2}\mathcal{\td L}S-\frac{1}{n}\left(\frac{1}{2}\mathcal{\td L}H^{2}\right)\\
= &\sum_{i,j,k}h_{ijk}^{2}+B\lagl x,N\ragl^2-\frac{1}{n}|\nabla H|^{2}+x^\top\left(4A^{2}+BI-\frac{4HA}{n}\right)(x^\top)^{t} +B^{2}+\frac{H^{2}B}{n}-\lambda B_{3}-\frac{2\lambda HB}{n}\\
\geq &\sum_{i,j,k}h_{ijk}^{2}-\frac{1}{n}|\nabla H|^{2}+B^{2}+\frac{H^{2}B}{n}-\lambda B_{3}-\frac{2\lambda HB}{n},
\end{align*}
where the assumption \eqref{eq1.4-2} has been used. Once again we use Lemma \ref{lem2.5} to get
\begin{align*}
|B_{3}|\leq \frac{n-2}{\sqrt{n(n-1)}}B^{\frac{3}{2}},
\end{align*}
where the equality holds if and only if at least $n-1 $ of $\mu_{i}$ are equal. It then follows that
\begin{align}
\frac{1}{2}\mathcal{\td L}B\geq &\sum_{i,j,k}h_{ijk}^{2}-\frac{1}{n}|\nabla H|^{2}+B^{2}+\frac{H^{2}B}{n} -|\lambda|\frac{n-2}{\sqrt{n(n-1)}}B^{\frac{3}{2}}-\frac{2}{n}\lambda HB\nnm\\
=&\sum_{i,j,k}h_{ijk}^{2}-\frac{1}{n}|\nabla H|^{2}+B\left(B+\frac{H^{2}}{n}-|\lambda|\frac{n-2}{\sqrt{n(n-1)}}B^{\frac{1}{2}} -\frac{2}{n}\lambda H\right)\nnm\\
=&\sum_{i,j,k}h_{ijk}^{2}-\frac{1}{n}|\nabla H|^{2}+B\left(\left(\sqrt{B}-|\lambda|\frac{n-2}{2\sqrt{n(n-1)}} \right)^{2} +\frac{1}{n}(H-\lambda)^{2}-\frac{n\lambda^{2}}{4(n-1)}\right).\label{eq3.1}
\end{align}
Because of \eqref{eq1.4-1}, we can apply the Corollary \ref{cor2.4} to functions $1$ and $B=S-\frac{H^{2}}{n}$ to find
\begin{align}\label{eq3.2}
0\geq&\int_{M^n}\left(\sum_{i,j,k}h_{ijk}^{2}-\frac{1}{n}|\nabla H|^{2}\right)e^{-\frac{\langle x, x\rangle^{2}}{4}}dV_{M^n} \\
&+\int_{M^n}B\left(\left(\sqrt{B}-|\lambda|\frac{n-2}{2\sqrt{n(n-1)}}\right)^{2}+\frac{1}{n}(H-\lambda)^{2}-\frac{n\lambda^{2}}{4(n-1)}\right)e^{-\frac{\langle x, x\rangle^{2}}{4}}dV_{M^n}.
\end{align}

If $B\not\equiv 0$ and, for all $p\in M^n$, \eqref{1.6} does not hold, that is
$$\left(\sqrt{S(p)-\frac{H^{2}(p)}{n}} -|\lambda|\frac{n-2}{2\sqrt{n(n-1)}}\right)^{2}+\frac{1}{n}(H(p)-\lambda)^{2} -\frac{n\lambda^{2}}{4(n-1)}\geq 0.
$$
everywhere on $M^n$, then the right hand side of \eqref{eq3.2} is nonnegative. It then follows that
\be\label{ad3.5}\sum_{i,j,k}h_{ijk}^{2}-\frac{1}{n}|\nabla H|^{2}\equiv 0,\ee
and at points where $B\neq 0$
\be\label{ad3.6}
\left(\sqrt{S(p)-\frac{H^{2}(p)}{n}} -|\lambda|\frac{n-2}{2\sqrt{n(n-1)}}\right)^{2}+\frac{1}{n}(H(p)-\lambda)^{2} -\frac{n\lambda^{2}}{4(n-1)}=0.
\ee
By \eqref{ad3.5} and \eqref{ad3.6}, the second fundamental form $h$ of $x$ is parallel. In particular, $x$ is isoparametric and thus both $B$ and $H$ are constant.
Since $B\neq 0$, the equality \eqref{ad3.3} shows that $x$ is a complete isoparametric space-like hypersurface in $\bbr^{n+1}_1$ of exactly two distinct principal curvatures one of which is simple. It then follows by \cite{HBL} and $B\neq 0$ that $x$ is isometric to one of the product spaces $\mathbb{H}^{n-1}(c)\times \bbr^1\subset\bbr^{n+1}_1$ and $\mathbb{H}^1(c)\times \bbr^{n-1}\subset\bbr^{n+1}_1$. But it is clear that, for both of these two product spaces, the function $\lagl x,x\ragl$ is not a constant so that both $\mathbb{H}^{n-1}(c)\times \bbr^1\subset\bbr^{n+1}_1$ and $\mathbb{H}^1(c)\times \bbr^{n-1}\subset\bbr^{n+1}_1$ could not be $\lambda$-hypersurfaces with $s=\lagl x,x\ragl$. This contradiction proves that either $B\equiv 0$, namely, $x$ is totally umbilical and isometric to either of the hyperbolic $n$-space $\mathbb{H}^n(c)\subset\bbr^{n+1}_1$ and the Euclidean $n$-space $\bbr^n\subset\bbr^{n+1}_1$, or there exists some $p\in M^n$ such that \eqref{1.6} holds.

The proof of Theorem \ref{thm 1.3} is thus finished.\endproof

\end{document}